\documentclass{article}
\usepackage{amsfonts,amsmath,amsthm,amssymb}
\usepackage{graphics, epsfig}
\usepackage{color}
\usepackage{appendix}
\usepackage{ulem}
\usepackage[makeroom]{cancel}
\usepackage{fancyhdr}
\usepackage{centernot}
\usepackage{mathtools}
\usepackage{ stmaryrd }

 \usepackage[usenames,dvipsnames]{pstricks}
 \usepackage{pst-grad} 
 \usepackage{pst-plot} 
\allowdisplaybreaks

\let\TeXchi\chi
\newbox\chibox
\setbox0 \hbox{\mathsurround0pt $\TeXchi$}
\setbox\chibox \hbox{\raise\dp0 \box 0 }
\def\chi{\copy\chibox}

\newtheorem{proposition}{Proposition}[section]
\newtheorem{theorem}{Theorem}[section]
\newtheorem{definition}{Definition}[section]

\newtheorem{remark}{Remark}[section]
\newtheorem{conjecture}{Conjecture}[section]

\numberwithin{equation}{section}
\numberwithin{theorem}{section}
\numberwithin{definition}{section}
\numberwithin{example}{section}
\numberwithin{proposition}{section}
\numberwithin{lemma}{section}
\numberwithin{remark}{section}
\DeclareMathOperator{\Ima}{Im}
\DeclareMathOperator{\Dom}{Dom}
\DeclareMathOperator{\Spec}{Spec}
\DeclareMathOperator{\Hom}{Hom}
\DeclareMathOperator{\id}{id}
\DeclareMathOperator{\Ob}{Ob}
\DeclareMathOperator{\Inj}{Inj}
\DeclareMathOperator{\Ind}{Ind}
\newcommand\blfootnote[1]{%
  \begingroup
  \renewcommand\thefootnote{}\footnote{#1}%
  \addtocounter{footnote}{-1}%
  \endgroup
}
\begin{document}
\title{Structured spaces: categories, sheaves, bundles, schemes and cohomology theories}
\author
{Manuel Norman}
\date{}
\maketitle
\begin{abstract}
\noindent In [1] we introduced the concept of structured space, which is a topological space that locally resembles some algebraic structures. In [2] we proceeded the study of these spaces, developing two cohomology theories. The aim of this paper is to define categories of structured spaces, (pre)sheaves with values in such categories, and generalised notions of vector bundles, ringed spaces and schemes. Then, we will construct (using some techniques and also the general method in Section 2 of [2]) various cohomology theories for these objects.
\end{abstract}
\blfootnote{Author: \textbf{Manuel Norman}; email: manuel.norman02@gmail.com
}
\section{Introduction}
We briefly recall some important concepts that will be used throughout the paper.
\begin{definition}\label{Def:1.1}
A functor $\mathcal{F}:Op(X) \rightarrow \mathbf{C}$ (often written $\mathcal{F}:X \rightarrow \mathbf{C}$, tacitly assuming that the actual domain is the category of open subsets of $X$, and not $X$ itself), where $X$ is a topological space and $\mathbf{C}$ is some category, is called a presheaf (with values in $\mathbf{C}$) if it is defined as follows:\\
(i) For each open $U \subseteq X$, there is an object $\mathcal{F}(U) \in \mathbf{C}$\\
(ii) For each inclusion of open subsets of $X$, $U \subseteq V$, there corresponds a morphism $\rho_{U,V}: \mathcal{F}(V) \rightarrow \mathcal{F}(U)$ in the category $\mathbf{C}$.
\end{definition}
We notice that, by definition of functor, $\mathcal{F}$ satisfies the following properties:\\
1) For every open $U \subseteq X$, $\rho_{U,U} = \id_{\mathcal{F}(U)}$\\
2) If we have the inclusions of open subsets $U \subseteq V \subseteq W$, $\rho_{U,V} \circ \rho_{V,W} = \rho_{U,W}$\\
$\mathcal{F}(U)$ is called 'section of $\mathcal{F}$ over $U$' (for $U$ open).\\
A sheaf is a presheaf which satisfies some additional conditions:
\begin{definition}\label{Def:1.2}
A sheaf $\mathcal{F}:X \rightarrow \mathbf{C}$ is a presheaf which satisfies the following axioms:\\
1) (gluing) For every open cover $\lbrace U_i \rbrace_{i \in I}$ of any open subset $U \subseteq X$, and for every family $\lbrace s_i \rbrace_{i \in I}$ with $s_i \in \mathcal{F}(U_i)$, if:
$$ \rho_{U_i \cap U_j, U_i} (s_i) = \rho_{U_i \cap U_j, U_j} (s_j)$$
then there is some $s \in \mathcal{F}(U)$ such that $\rho_{U_i,U} (s) = s_i$ $\forall i \in I$;\\
2) (locality) For any $s,t \in \mathcal{F}(U)$, if
$$ \rho_{U_i,U}(s) =\rho_{U_i,U} (t), \quad \forall i \in I $$
then $s=t$.
\end{definition}
We define the functor $\Gamma(U, \mathcal{F}):= \mathcal{F}(U)$. Thanks to this functor, we can construct a sheaf starting from a presheaf (sheafification). Actually, there are various ways to sheafify a presheaf: here we will only consider the one below, but we however notice that, for instance, Artin, Grothendieck and Verdier showed in [3] another method to do this. The stalk $\mathcal{F}_x$ of a presheaf $\mathcal{F}$ on $X$ at $x \in X$ is the direct limit (see, for instance, [4-6] for some examples; see [7] for this particular situation) of the sections $\mathcal{F}(U)$ for $U$ containing $x$, i.e.:
$$ \mathcal{F}_x := \lim_{\longrightarrow} (\mathcal{F}(U))_{U \ni x} $$
We can now define the space 
$$S \mathcal{F}:= \coprod_{x \in X} \mathcal{F}_x$$
and endow it with with the disjoint union topology (we refer to [7] for the precise definition of this topology). This space is called the \'etale space of the presheaf. We indicated the disjoint union using $\coprod$ instead of $\bigsqcup$ to stress the fact that we are using the categorical dual of the product space construction (and since we are in the context of category theory, we prefer to use this symbol). Let $p: S \mathcal{F} \rightarrow X$ be the projection map defined by $p(\alpha):=x$ for all $\alpha \in \mathcal{F}_x$. The following $\widetilde{\mathcal{F}}$ is a sheaf:
$$ \widetilde{\mathcal{F}}:= \Gamma[S \mathcal{F}, p] $$
where $\Gamma[A,f]$ is the sheaf of continuous sections of $f$ on $B$, with $f:A \rightarrow B$ continuous and surjective, defined by (for every open $U \subseteq B$):
$$ \Gamma[A,f](U):= \lbrace g: U \rightarrow A \, | \, f \circ g = \id \, \text{and} \, g \, \text{is continuous} \rbrace $$
This sheaf is called the sheafification of $\mathcal{F}$. If $\mathcal{F}$ is a sheaf, $\widetilde{\mathcal{F}}$ is isomorphic to it. This construction is really useful in many situations. We refer to [13-15] and [56] for more on presheaves and sheaves.\\
We proceed recalling the definition of vector bundle, which is a fiber bundle where the fibers $\pi^{-1}(\lbrace x \rbrace)$ are vector spaces, and where a certain linearity condition is satisfied. The precise definition is as follows:
\begin{definition}\label{Def:1.3}
Consider $(E, X, \pi)$, where $E$ is the total space, $B$ is the base space, and they are all topological spaces; $\pi : E \rightarrow X$ is continuous and surjective. $(E, X, \pi)$ is a vector bundle if the following conditions are satisfied:\\
(i) $\forall x \in X$, $\pi^{-1}(\lbrace x \rbrace)$ is endowed with the structure of a finite dimensional vector space over some fixed field $\mathbb{F}$;\\
(ii) $\forall y \in X$, there is an open neighborhood $U \subseteq X$ of $y$, some $k \in \mathbb{N}$, and a homeomorphism $h: U \times \mathbb{F}^k \rightarrow \pi^{-1}(U)$ such that, $\forall x \in U$:
$$ (\pi \circ h)(x,z)=x \quad \forall z \in \mathbb{F}^k $$
and the map $s:\mathbb{F}^k \rightarrow \pi^{-1}(\lbrace x \rbrace)$ given by $z \xmapsto{s} h(x,z)$ is a linear isomorphism.
\end{definition}
The next notion is a really general one, and it includes as special cases many well known concepts (e.g. topological spaces (with the sheaf of continuous functions) and smooth manifolds (with the sheaf of smooth functions)):
\begin{definition}\label{Def:1.4}
Let $X$ be a topological space and let $\mathcal{O}_X$ be a sheaf of rings on $X$. Then, $(X, \mathcal{O}_X)$ is called a ringed space.\\
A locally ringed space is a ringed space such that all the stalks of its sheaf of rings are local rings.
\end{definition}
Thanks to ringed and locally ringed spaces, we can define schemes (see, for instance, [8-12]):
\begin{definition}\label{Def:1.5}
An affine scheme is a locally ringed space which is isomorphic to $\Spec(R)$ for some commutative ring $R$, where $\Spec(R)$ is the spectrum of $R$, i.e. the set of all prime ideals of $R$, endowed with the Zariski topology (the closed sets are the ones of the prime ideals containing $I$, for any ideal $I$ of $R$). This isomorphism has to be seen as an isomorphism of locally ringed spaces, because $\Spec(R)$ is also augmented with a structure sheaf which makes it a locally ringed space.\\
A scheme is a locally ringed space that can be covered by open sets $U_j$ such that each $U_j$ (viewed as a locally ringed space) is an affine scheme.
\end{definition}
We conclude by briefly describing structured spaces (see [1]):
\begin{definition}\label{Def:1.6}
A topological space $X$ is a structured space if for every $p \in X$ there is a neighborhood $U_p$ of $p$ which can be endowed with an algebraic structure. More precisely, a structured space is a couple $(X, f_s)$, where $f_s$ is the structure map which assigns to each \emph{fixed} neighborhood $U_p$ (we fix one of the possible neighborhoods and also its structure) a fixed algebraic structure $f_s(U_p)$.
\end{definition}
One of the reason why we introduced structured spaces is that many important spaces cannot be endowed \textit{globally} with some algebraic structure, but they may be endowed \textit{locally} with various local structures. This way, we can use algebraic tools to \emph{locally} study these spaces. Many examples can be found in Section 1 of [1].
\section{Categories of structured spaces}
We start constructing some categories of structured spaces. The idea is to consider as objects the structured spaces that have the same cardinality and the same local structures. More precisely, let $X$ and $Y$ be structured spaces. Consider any $_X U_p \in \mathcal{U}_X$. Suppose there is a fixed neighborhood in $Y$ (which can be indicated, after reordering if necessary, as $_Y U_p \in \mathcal{U}_Y$) which is endowed with the same algebraic structure of $_X U_p$, i.e. $f_s(_X U_p) \equiv f_s(_Y U_p)$ (recall the difference between '$=$' and '$\equiv$' when dealing with structure maps; see Section 2 in [1], at the end, for the definition of '$\equiv$'). We always want to deal with such situations; therefore, we will consider $X$ and $Y$ to be in the same category if there is a bijection from $\mathcal{U}_X$ to $\mathcal{U}_Y$ such that to each fixed neighborhood in $X$ there corresponds one and only one fixed neighborhood of $Y$ which is endowed with the same algebraic structures (i.e. their structure maps are equivalent under '$\equiv$'). We can say that $X$ and $Y$ have the same cardinality (even if they have an infinite number of fixed neighborhoods, the bijection can be used to mean 'same cardinality' (recall Cantor's work on cardinality of infinite sets)) and that they have the same local structures under the chosen bijection. Notice that it could be possible to have more than one such bijection: for instance, if two fixed neighborhoods of $X$ (and $Y$) are rings, it is clear that there are more possible bijections satisfiyng the above condition. We will thus have a dependence on the bijection, which does not change the spaces belonging to the same category; it will however change the morphisms, as we will now see. We can start to construct our category, which is indicated by $\textbf{Str}(h,f_s)$, where the objects are all the structured spaces for which, if $X$ and $Y$ are any of them, there is a bijection $h_p$ in the family of chosen bijections $h=(h_t)$ such that each fixed neighborhood of $X$ is sent to a fixed neighborhood of $Y$ (and viceversa, since we have a bijection) in such a way that the neighborhoods have an equivalent (under '$\equiv$') algebraic structure w.r.t. $f_s$ (of course, we can use the same $f_s$ up to equivalence '$\equiv$'; in general, they are different ('$\neq$'), but since this in not confusing, we will simply use $f_s$). Note that $X$ and $Y$ are thus considered to be structured spaces w.r.t. the same (up to equivalence) $f_s$, so it is more correct to write $(X,f_s)$ and $(Y,f_s)$. Thanks to the family of bijections $h$, we can write the following (after reordering) for all the spaces in \textbf{Str}$(h,f_s)$:
$$ ... \xleftrightarrow{h_j} U_p \xleftrightarrow{h_t} V_p \xleftrightarrow{h_r}  W_p \xleftrightarrow{h_i} ... $$
for all $p$, where $U_p$ is a fixed neighborhood in $X$, $V_p$ is a fixed neighborhood in $Y$, $W_p$ is a fixed neighborhood in $Z$, ...\\
We give a simple example (where we reorder the fixed neighborhoods and the family $h$ so that the indices are all the same):
$$ \text{ring} \, U_1 \, \xleftrightarrow{h_1} \, \text{ring} \, V_1 $$
$$ \text{unital magma} \, U_2 \, \xleftrightarrow{h_2} \, \text{unital magma} \, V_2 $$
$$ \text{topological group} \, U_3 \, \xleftrightarrow{h_3} \, \text{topologial group} \, V_3 $$
Now that we have defined the objects of the category $\textbf{Str}(h,f_s)$ for some $h$ and $f_s$, we need to define morphisms. We recall the definition of homomorphisms of structured spaces given in [1]. A homomorphism $f$ between two structured spaces $X$, $Y$ in $\textbf{Str}(h,f_s)$ is a collection of homomorphisms, where each homomorphism of this family has, as domain, a fixed neighborhood in $X$, and it is mapped to the corresponding (w.r.t. $h$) fixed neighborhood in $Y$. The family contains one and only one homomorphism with domain some fixed neighborhood (i.e. the family contains only one homomorphism for each correspondence, w.r.t. $h$, of fixed neighborhoods). The class of morphisms from $X$ to $Y$ (both in $\textbf{Str}(h,f_s)$), indicated by $\Hom(X,Y)$, is the set of all homomorphisms from $X$ to $Y$ (as defined above). The composition of two homomorphisms $f:X \rightarrow Y$, $g: Y \rightarrow Z$ is defined as the composition of each homomorphism in the family, that is, if $f=(f_p)$ and $g=(g_p)$, where we have already reordered $f$ and $g$ so that $U_p \xrightarrow{f_p} V_p \xrightarrow{g_p} W_p$, we have $g \circ f=(g_p \circ f_p)$. Here, as usual in this context (see, for instance, the categories \textbf{Grp}, \textbf{Ab}, ...), we use a more general definition of composition, where we only require $\Ima f_p \subseteq \Dom g_p = Y$ for all $p$, and we do not necessarily have the equality of the image and the domain. It only remains to check that what we have constructed is indeed a category:
\begin{proposition}\label{Prop:2.1} 
Consider some $h$ and $f_s$ as above. Then, $\textbf{Str}(h,f_s)$ is a category.
\end{proposition}
\begin{proof}
We only have to verify that associativity holds, and that there is an identity element. Consider $f \circ (g \circ h)$. After reodering, we can write $f \circ (g \circ h) = (f_p \circ (g_p \circ h_p))$. But since these are all homomorphisms between the same kind of algebraic structure, we known that they are all equal to $((f_p \circ g_p) \circ h_p)=(f \circ g) \circ h$. Thus, associativity holds. Moreover, for every $X \in \textbf{Str}(h,f_s)$, we have a homomorphism $\id_X : X \rightarrow X$ in $\Hom(X,X)$ defined by $\id_X=(\id_{U_p})$ for all the fixed $U_p$'s in $X$, such that $\id_X \circ f =f$ and $g \circ \id_X = g$ whichever are $f: A \rightarrow X$, $g: X \rightarrow B$. This equality follows because the same equation clearly holds for each $\id_{U_p}$.
\end{proof}
There are some particular kinds of these categories that will be used to construct some cohomologies involving sheaves, namely $\textbf{Part}(h,f_s)$, which are categories of partitionable spaces (see [1]). More precisely, this time the objects are partitionable spaces (which are particular cases of structured spaces). The category is then constructed similarly as above.\\
We conclude this section defining (pre)sheaves with values in these categories. Actually, these are simply defined as in Section 1. We only notice that, for reasons that will be clear in the next section (namely, when we decompose a presheaf with values in \textbf{Str}$(h,\textbf{Ab})$ to some presheaves with values in \textbf{Ab}, we would have some problems with property 2) in the definition of functor (see Section 1); indeed, it would be possible to obtain something which is not anymore a functor after the decomposition, so we avoid this possibility as follows), when dealing with some structured cohomologies we will use sheaves with values in \textbf{Part}$(h,f_s)$, and not in \textbf{Str}$(h,f_s)$ (which will be considered later in this paper). This is why we have also defined these categories consisting of partitionable spaces.
\section{Structured sheaf and \v{C}ech cohomologies}
From now on, we will mainly follow these steps in each section:\\
$\bullet$ we briefly recall how the cohomologies we are dealing with are defined, and if necessary we define a generalisation of the concepts involved\\
$\bullet$ we generalise everything to the structured case, which usually consists in finding a way to decompose the considered object into its fixed neighborhoods and then apply the methods in Section 2 of [2] (Section 5 is the unique exception).\\
The idea behind the decomposition process is essentialy the same, so we will explicitely describe it only in this section. In the next ones we will consider the parts which differ from each other, otherwise it would be repetitive.\\
We start constructing structured sheaf cohomology. Recall that (see, for instance, [7] and [13-14]), given a sheaf $\mathcal{F}$ with values in \textbf{Ab} (or given a presheaf of abelian groups that we immediately sheafify), we can consider some injective resolution (this is possible because the category of sheaves of abelian groups is abelian and has enough injectives):
$$ 0 \rightarrow \mathcal{F} \rightarrow \mathcal{I}_0 \, \rightarrow \mathcal{I}_1 \, \rightarrow ...$$
Then, sheaf cohomology groups are defined as the usual cohomology groups of the following cochain complex:
$$ 0 \rightarrow \mathcal{I}_0(X) \rightarrow \mathcal{I}_1(X) \rightarrow \mathcal{I}_2(X) \rightarrow ... $$
where we have applied the functor $\Gamma(X, \cdot)$ (the coboundary maps are obtained by applying this functor to the maps in the resolution above). The cohomology groups are the right derived functors of $\Gamma(X, \cdot)$, and they are independent of the chosen resolution.\\
Now consider a presheaf $\mathcal{F}:X \rightarrow \textbf{Part}(h,\textbf{Ab})$, where $\textbf{Ab}$ in place of $f_s$ means that all the fixed neighborhoods are abelian groups. We can decompose \footnote{Notice that, here, when we say 'decomposition' we do not mean that their union is the former object: we only mean that from a certain object we can obtain some other objects by considering "some parts of it" in such a way that we completely describe, from that point of view, the object via such parts.} each $\mathcal{F}(U)$, which is a structured space, into its fixed neighborhoods, say $(\mathcal{F}(U))_p$, which are abelian groups, as said above. Define $\mathcal{F}_p (U)$, for each open $U$ in $X$, to be $(\mathcal{F}(U))_p$. Since all the fixed neighborhoods are abelian groups, $\mathcal{F}_p$ has values in \textbf{Ab}. We wonder whether it is a presheaf or not. It is clear that to each $U$ we assign $(\mathcal{F}(U))_p$, and to each inclusion $U \subseteq V$ we assign the map (reorder the $p$'s using $h$: if two fixed neighborhoods of two structured spaces are associated via $h$, then they must be indicated with the same $p$, and viceversa, if two fixed neighborhoods are denoted by the same $p$, they must be related by $h$; this will always be the meaning of 'reordering' when 'decomposing' some object) $f_p : \mathcal{F}_p(V) \rightarrow \mathcal{F}_p(U)$, which is the homomorphism $f_p$ (after an obvious reordering) in the family $f=(f_t)$, where $f$ is the homomorphism assigned to $U \subseteq V$ via $\mathcal{F}$. We only have to verify that $\mathcal{F}_p$ are functors. It is clear that $\rho^p _{U,U} = \id_{\mathcal{F}_p(U)}$, because this was true with $\mathcal{F}$ by definition of presheaf. Moreover, it is not difficult to see that, if $U \subseteq V \subseteq W$, $\rho^p _{U,V} \circ \rho^p _{V,W} = \rho^p _{U,W}$. Notice that this is always true because by assumption the spaces are partitionable; otherwise, this would not be necessarily true. Also note that the fact that we use partitionable spaces is not so restrictive, from some points of views: for instance, recall all the partitionable spaces defined in Example 1.3 of [1] (those ones are not all abelian groups, but it is clear that there are many interesting example satisfying this property). Now that we have verified that the $\mathcal{F}_p$'s are presheaves, we can sheafify all of them and obtain the sheaves $\widetilde{\mathcal{F}_p}$. We can construct the following square/rectangle (depending on the number of $p$'s, which is assumed to be at most countable infinite) using the cochains of the sheaf cohomologies of each $\widetilde{\mathcal{F}_p}$:
$$0 \rightarrow \, _1 \mathcal{I}_0(X) \xrightarrow{ _0 \partial^0} \, _1 \mathcal{I}_1(X) \xrightarrow{ _0 \partial^1} \, _1 \mathcal{I}_2(X) \xrightarrow{ _0 \partial^2} ...$$
$$\downarrow \, _0 \widetilde{\partial}^0 \qquad \quad \downarrow \, _0 \widetilde{\partial}^1  \qquad  \, \, \downarrow \,_0 \widetilde{\partial}^2 $$
$$0 \rightarrow \, _2 \mathcal{I}_0(X) \xrightarrow{_1 \partial^0} \, _2 \mathcal{I}_1(X) \xrightarrow{_1 \partial^1} \, _2 \mathcal{I}_2(X) \xrightarrow{_1 \partial^2} ...$$
$$\downarrow \, _1 \widetilde{\partial}^0  \qquad \quad \downarrow \, _1 \widetilde{\partial}^1 \qquad \, \, \downarrow \,_1 \widetilde{\partial}^2 $$
$$0 \rightarrow \, _3 \mathcal{I}_0(X) \xrightarrow{_2 \partial^0} \, _3 \mathcal{I}_1(X) \xrightarrow{_2 \partial^1} \, _3 \mathcal{I}_2(X) \xrightarrow{_2 \partial^2} ...$$
$$\downarrow \, _2 \widetilde{\partial}^0 \qquad \quad \downarrow \, _2 \widetilde{\partial}^1 \qquad \, \, \downarrow \,_2 \widetilde{\partial}^2 $$
$$\, ... \qquad \qquad \quad  \, ... \qquad \qquad  \, ...  \qquad $$
where $_n \partial^p$ are given by the sheaf cohomologies, while  $_n \widetilde{\partial}^p$ are a chosen family of vertical homomorphisms. As explained in [2], we can choose any such family, only requiring it to satisfy some conditions (see Theorem 2.1 and Theorem 2.2 in [2]); it is always possible to choose the trivial family, where all the vertical homomorphisms are $\equiv 0$ (as discussed in [2], this does not cause problems to the cohomology theory, which is still interesting to study). Then, we can consider their square/rectangular cohomology, or, if the square/rectangle above turns out to be a double complex, we can also consider the cohomology groups $H^{p,q}$ of the double complex, or the cohomology groups of the total complex obtained from the double complex. All these possibilities lead to some cohomology theory for the starting presheaf $\mathcal{F}:X \rightarrow \textbf{Str}(h,\textbf{Ab})$. They will be all called 'structured sheaf cohomologies'. An important question now is: are all these cohomologies independent of the chosen resolutions? For the cohomology groups arising from the horizontal lines, this is clearly true, because the cohomology groups will be the ones of the sheaf cohomology of some $\widetilde{\mathcal{F}_p}$; the problem is when we have the vertical homomorphisms. Unfortunately, it is not always true that the cohomology groups that come out also from the vertical homomorphisms are independent of the resoultions: for instance, if we have the trivial family, it is clear (see also Proposition 2.1 in [2]) that some cohomology groups could be equal to some $_n \mathcal{I}_p(X)$ or to some kernel of the horizontal homomorphisms, which generally depend on the resolution. However, we can avoid this problem as follows. We consider the case where all the vertical homomorphisms are trivial and we use square/rectangular cohomology; other cases can be treated similarly. We first recall the following construction of injective resolutions for objects in a category with enough injectives (as, for instance, the category of sheaves of abelian groups):\\
1) Let $X$ be any object (in our case, a sheaf) in the category. We first take some injective sheaf $_0 J_1$ and a monomorphism from $X$ to this injective object. We know that such an injetive object exists by definition of category with enough injectives.\\
2) Suppose that we have another (different) injective sheaf $_0 J_2$ such that $_0 J_1(Y) \subseteq \, _0 J_2(Y)$ (where $Y$ is the domain of the sheaf $X$ and of all the other sheaves considered here), with a monomorphism from $X$ to $_0 J_2$. If it does not exist, there is no problem: consider only the first one, and in what follows ignore the steps for $_0 J_2$.\\
3) Both these injective sheaves are such that there is, for each of them, an injective sheaf $_1 J_1$ ($_1 J_2$, respectively) with monomorphism from the previous object to this one. Again, suppose that there is (for each of them) another injective sheaf $_1 \widehat{J}_1$ ($_1 \widehat{J}2$, respectively) such that $_1 J_1(Y) \subseteq \, _1 \widehat{J}_1(Y)$ (and similarly for the other one). Continue this way, and obtain (these are only the first terms):
$$\underline{_1 J_1}$$
$$ \nearrow \qquad \quad$$
$$\underline{_0 J_1} \rightarrow \, _1 \widehat{J}_1 \qquad \qquad$$
$$ \nearrow \qquad \qquad \qquad \qquad$$
$$X \qquad \qquad \qquad \qquad \qquad \qquad$$
$$ \searrow \qquad \qquad \qquad \qquad \qquad$$
$$\underline{_0 J_2} \rightarrow \, _1 J_2 \qquad \qquad \qquad$$
$$\searrow \qquad \qquad \qquad$$
$$\underline{_1 \widehat{J}_2} \qquad $$
and so on. It is then clear that all the sheaves on the diagonal starting from $X$ and going down contains (after applying the functor $\Gamma(Y, \cdot)$) the corresponding elements on the diagonal starting from $X$ and going up (the elements of these diagonals are underlined). It is also reasonable to assume that the kernel of a map from an element of some injective resolution and the next element of that resolution is contained in the kernel of the corresponding map in the other resolution where all the elements contain the objects of the first one. This suggest the following definition ($_j d^p$ is the $p$-th coboundary map in the resolution $\mathcal{I}^j$):
\begin{definition}\label{Def:3.1}
Given two injective resolutions of a sheaf of abelian groups, say $\mathcal{I}^1=\mathcal{I}^1 _{\bullet}$ and $\mathcal{I}^2=\mathcal{I}^2 _{\bullet}$, we define $\prec$ as follows:
$$\mathcal{I}^1 \prec \mathcal{I}^2 \Leftrightarrow \mathcal{I}^1 _{p}(Y) \subseteq \mathcal{I}^2 _{p}(Y) \, \text{and} \, \ker \, _i d^p \subseteq \ker \, _j d^p, \, \, \forall p$$
\end{definition}
It is clear that $\prec$ is a preorder, and hence the collection $I:=\lbrace \, _n \mathcal{I}^j \rbrace$  ($n$ indicates that we are dealing with $\widetilde{\mathcal{F}}_n$) is a directed set. We want to take the direct limit of the cohomology groups of the structured sheaf cohomology to "refine" them. This way, they will be independent of the chosen resolution. There are three kinds of cohomology groups arising when we use the trivial family of vertical homomorphisms:\\
(i) the sheaf cohomology groups of some $\widetilde{\mathcal{F}}_p$\\
(ii) $_n \mathcal{I}^j _p(Y)$ (the $p$-th term of some resolution $_n \mathcal{I}^j$ evaluated at $Y$)\\
(iii) $\ker \, _i \partial(p,j)$, where $_i \partial(p,j)$ indicates some horizontal coboundary map\\
(for the last two groups, see Proposition 2.1 in [2], and adapt it to the above case). In each of the above cases we want to find a direct system over $I$. We start with (i): the collection $ \lbrace H^p(_n \mathcal{I}^j) \rbrace_I$ ($n$ and $p$ are fixed) is the collection of the $p$-th sheaf cohomology groups w.r.t. $\widetilde{\mathcal{F}}_n$ and w.r.t. some resolution $_n \mathcal{I}^j$. Actually, we know that these groups are independent of the chosen resolution, so there is only one element in this collection. We can thus define a homomorphism $f_{i,j}:H^p(_n \mathcal{I}^i) \rightarrow H^p(_n \mathcal{I}^j)$ as the identity homomorphism $\id_{H^p(_n \mathcal{I}^i)}$. This clearly satisfies all the necessary properties, and thus we have a direct system. In case (ii), we have to find a homomorphism $f_{i,j}: \, _n \mathcal{I}^i _p(Y) \rightarrow \, _n \mathcal{I}^j _p(Y)$ ($n$ and $p$ fixed) for all $_n \mathcal{I}^i \prec \, _n \mathcal{I}^j$. But we know by definition that $_n \mathcal{I}^i _p(Y) \subseteq \, _n \mathcal{I}^j _p(Y)$, so we can consider the inclusion homomorphism  which assigns to each element of the first group the same element in the bigger group. It is again clear that the needed properties are satisfied (recall that we have defined composition so that we only need to have the image \textit{contained} in the domain, and not necessary equal to the domain). Thus, we have a direct system. It only remains to find homomorphisms for (iii). Since we have assumed that the kernels in a resolution are contained in the corresponding kernels of the resolution on the right of the sign $\prec$, it is clear that we can again take the inclusion map as in (ii). Thus, we have direct systems in each case, and therefore we can take the direct limits to obtain the "refined" cohomology groups. It is clear that the direct limit of a group in (i) is the same as the groups because of independence. Using this method, we have arrived at a (square/rectangular) cohomology theory which does not depend on the chosen resolution. We can summarise everything in the following important result:
\begin{theorem}[Structured sheaf cohomology]\label{Thm:3.1}
Consider a presheaf $\mathcal{F}:X \rightarrow \textbf{Part}(h,\textbf{Ab})$. We can decompose this presheaf into $\mathcal{F}_p:X\rightarrow \textbf{Ab}$, which are presheaves of abelian groups. Sheafify them and obtain the sheaves $\widetilde{\mathcal{F}_p}$. Consider some cochains of their sheaf cohomologies (w.r.t. some resolutions) and construct a square/rectangle as usual (see [2]). We can choose any possible kind of vertical homomorphisms (for instance, the trivial family) so that we have a square/rectangular cohomology, or, if the square/rectangle turns out to be a double complex, also the cohomology of the double complex or the cohomology of the corresponding total complex. In any of these cases, we have a structured sheaf cohomology. We can refine it (obtaining the independence from the resolutions) by taking the direct limits of the cohomology groups $H^p(X, \mathcal{F})$:
$$ \widetilde{H}^p (X, \mathcal{F}):= \lim_{\longrightarrow_{I}} H^p(X, \mathcal{F}) $$
where $I$ is a directed set under $\prec$, which is defined accordingly to the situation (for instance, see Definition \ref{Def:3.1} for the particular case considered there).
\end{theorem}
Something similar can be done with \v{C}ech cohomology. Recall that \v{C}ech cohomology groups are defined for a sheaf $\mathcal{F}$ of abelian groups on $X$ and for an open cover $\mathcal{V}=(\mathcal{V}_i)_i$ as follows:
$$ C^p(X, \mathcal{F}):= \prod_{(i_1, ..., i_p) \in J^{p+1}} \mathcal{F}(\mathcal{V}_{i_0} \cap \mathcal{V}_{i_1} \cap ... \cap \mathcal{V}_{i_p})$$
The intersection in the argument of the presheaf is usually denoted by $V_{i_0 i_1 ... i_p}$. We refer to [7] and [13-14] for more on this cohomology. We can define a refinement for the covers, take the direct limit and then end up with "refined" \v{C}ech cohomology, which does not depend on the chosen open cover.\\
We construct structured \v{C}ech cohomology as before: take a presheaf $\mathcal{F}:X \rightarrow \textbf{Part}(h, \textbf{Ab})$ and decompose it in $\mathcal{F}_p: X \rightarrow \textbf{Ab}$, which are presheaves of abelian groups (the proof is the same as the one with structured sheaf cohomology). Sheafify all these presheaves and obtain the sheaves $\widetilde{\mathcal{F}_p}$. For each of these sheaves, consider some cover $\mathcal{V}_p$ and construct a square/rectangle as usual (see [2]). Then we can consider square/rectangular cohomology (or, if we have a double complex, we can also consider its cohomology or the cohomology of the total complex associated to it) and get 'structured \v{C}ech cohomology'. At the end, in order to have cohomology groups which are independent of the chosen covering, take the direct limits over the refinements of the open covers to obtain "refined" structured \v{C}ech cohomology (for the two cohomology groups arising from Proposition 2.1 in [2], we can proceed in a similar way to structured sheaf cohomology above). Thus, a result which is completely analogous to Theorem \ref{Thm:3.1} holds:
\begin{theorem}[Structured \v{C}ech cohomology]\label{Thm:3.2}
Consider a presheaf $\mathcal{F}:X \rightarrow \textbf{Part}(h,\textbf{Ab})$. We can decompose this presheaf into $\mathcal{F}_p:X\rightarrow \textbf{Ab}$, which are presheaves of abelian groups. Sheafify them and obtain the sheaves $\widetilde{\mathcal{F}_p}$. Consider some cochains of their \v{C}ech cohomologies and construct a square/rectangle as usual (see [2]). We can choose any possible kind of vertical homomorphisms (for instance, the trivial family) so that we have a square/rectangular cohomology, or, if the square/rectangle turns out to be a double complex, also the cohomology of the double complex or the cohomology of the corresponding total complex. In any of these cases, we have a structured \v{C}ech cohomology. We can refine it by taking the direct limits of the cohomology groups $H^p(\mathcal{V}, \mathcal{F})$:
$$ \widetilde{H}^p (X, \mathcal{F}):= \lim_{\longrightarrow_{\mathcal{V}}} H^p(\mathcal{V}, \mathcal{F}) $$
where the limit is taken over the refinements of the open covers.
\end{theorem}
\section{Structured Hochschild cohomology}
In this section we define a generalisation of the concepts of ringed and locally ringed spaces, and then we analyse a structured version of Hochschild cohomology for these spaces. The idea is to have, similarly to the previous section, a sheaf with values in the category of partitionable spaces whose neighborhoods are rings/local rings. However, here we also present another possibility: a topological space which can be written as a union of ringed/locally ringed spaces. We can also construct Hochschild cohomology for structural schemes; we will do this at the end of the section.\\
We start with the former option. First of all, we recall the definition of Hochschild cohomology for ringed spaces \footnote{Hochschild (co)homology can be defined for various structures. Here we will use it in relation to ringed spaces and schemes; for other examples, see [30-33].}. We mainly refer to [16] (see also [17-29]). Consider a field $\mathbb{K}$ (actually, we could be more general and consider a commutative base ring). We define the Hochschild cohomology groups for a $\mathbb{K}$-linear category $\mathcal{A}$ as:
$$ C^p(\mathcal{A}):= \prod_{A_0, ..., A_p \in \Ob(\mathcal{A})} \Hom_{\mathbb{K}} (\mathcal{A}(A_{p-1}, A_p) \otimes_{\mathbb{K}} \cdot \cdot \cdot \otimes_{\mathbb{K}} \mathcal{A}(A_{0}, A_1), \mathcal{A}(A_{0}, A_p)) $$
Now consider a $\mathbb{K}$-linear abelian category $\mathcal{B}$. Then, we define:
$$ C^p _{ab}(\mathcal{B}) = C^p (\Inj \Ind (\mathcal{B})) $$
The Hochschild cohomology for $\mathbb{K}$-linear ringed spaces $(X, \mathcal{O}_X)$ (non-commutative ringed spaces are also allowed) is:
$$ C^p(X):= C^p _{ab} (\textbf{Mod}(X))$$
where \textbf{Mod}$(X)$ is the category of sheaves of right modules over $X$. Here, since this category turns out to have enough injectives, we have used Theorem 6.6 in [16] to obtain: $C^p _{ab} (\textbf{Mod}(X)) \cong C^p _{sh} (\Inj \textbf{Mod}(X))$. We will use this definition for the Hochschild complex. Many important and interesting considerations follow from this definition; in particular, under some assumptions these cohomology groups are isomorphic to some "apparently different" Hochschild cohomologies considered by other authors. We will not go deeper here, see Section 7 in [16] for more details. We notice that, even though not explicitely written above, $C^p(X)$ actually means $C^p(X, \mathcal{O}_X)$. Since we will need later to deal with different sheaves over the same space, we will specify them in order to avoid confusion.\\
We define the following generalised notion:
\begin{definition}\label{Def:4.1}
A structural space $(X,\mathcal{F})$ is a topological space $X$ together with a sheaf $\mathcal{F}:X \rightarrow \textbf{Str}(h,f_s)$.
\end{definition}
Clearly, since \textbf{Part}$(h,f_s)$ is a particular case of \textbf{Str}$(h,f_s)$, this definition also includes the following:
\begin{definition}\label{Def:4.2}
A structural ringed space is a structural space $(X,\mathcal{F})$ with $\mathcal{F}:X \rightarrow \textbf{Part}(h, \textbf{Ring})$ \footnote{For example, $\bigcup_{n=1}^{k} \mathbb{R}^n$, with $k$ possibly infinite, is a partitionable space whose fixed nieghborhoods can be all endowed with the structure of a ring (define both addition and multiplication component-wise: $(a_1, ..., a_n) + (b_1, ..., b_n):=(a_1+b_1, ..., a_n + b_n)$ and $(a_1, ..., a_n) \cdot (b_1, ..., b_n):=(a_1 b_1, ..., a_n b_n)$, which is quite common when dealing with the ring $\mathbb{F}^n$, for some field $\mathbb{F}$).}, where \textbf{Ring} is the category of rings.
\end{definition}
Actually, as previously said, the following definition is also natural:
\begin{definition}\label{Def:4.3}
A semi-structural ringed space is a topological space which is the union of some ringed spaces.
\end{definition}
Notice that we do not require a semi-structural space $X$ to be neither a structured nor a structural space. We will verify, however, that under some conditions we obtain the same Hochschild cohomologies from these two definitions.\\
Now consider a structural ringed space $(X,\mathcal{F})$. Decompose the sheaf $\mathcal{F}:X \rightarrow \textbf{Part}(h,\textbf{Ring})$ into presheaves $\mathcal{F}_p : X \rightarrow \textbf{Ring}$ (as in Section 3, but this time with \textbf{Ring}). Notice that $(X,\widetilde{\mathcal{F}_p})$ is a ringed space for each $p$ (the tilde indicates the sheafification, as usual). Then, assuming that the number of $p$'s is at most countable, we can form a square/rectangle with the Hochschild cohomologies of these spaces:
$$0 \rightarrow \, C^0(X,\widetilde{\mathcal{F}}_1) \xrightarrow{ _0 \partial^0} \, C^1(X,\widetilde{\mathcal{F}}_1) \xrightarrow{ _0 \partial^1} \, C^2(X,\widetilde{\mathcal{F}}_1) \xrightarrow{ _0 \partial^2} ...$$
$$\downarrow \, _0 \widetilde{\partial}^0 \qquad \quad \downarrow \, _0 \widetilde{\partial}^1  \qquad  \, \, \downarrow \,_0 \widetilde{\partial}^2 $$
$$0 \rightarrow \, C^0(X,\widetilde{\mathcal{F}}_2) \xrightarrow{_1 \partial^0} \, C^1(X,\widetilde{\mathcal{F}}_2) \xrightarrow{_1 \partial^1} \, C^2(X,\widetilde{\mathcal{F}}_2) \xrightarrow{_1 \partial^2} ...$$
$$\downarrow \, _1 \widetilde{\partial}^0  \qquad \quad \downarrow \, _1 \widetilde{\partial}^1 \qquad \, \, \downarrow \,_1 \widetilde{\partial}^2 $$
$$0 \rightarrow \, C^0(X,\widetilde{\mathcal{F}}_3) \xrightarrow{_2 \partial^0} \, C^1(X,\widetilde{\mathcal{F}}_3) \xrightarrow{_2 \partial^1} \, C^2(X,\widetilde{\mathcal{F}}_3) \xrightarrow{_2 \partial^2} ...$$
$$\downarrow \, _2 \widetilde{\partial}^0 \qquad \quad \downarrow \, _2 \widetilde{\partial}^1 \qquad \, \, \downarrow \,_2 \widetilde{\partial}^2 $$
$$\, ... \qquad \qquad \quad  \, ... \qquad \qquad  \, ...  \qquad $$
where $_n \partial^p$ are the usual coboundary maps of Hochschild cohomology, and the vertical homomorphisms can be taken, for instance, to be the trivial ones (or, more generally, any kind of family satisfying the necessary conditions, see Theorem 2.1, Theorem 2.2 and Remark 2.1 in [2]). This cohomology is called 'structured Hochschild cohomology (for structural ringed spaces)'. Another kind of cohomology theory similar to this one is 'structured Hochschild cohomology for semi-structural ringed spaces'. Suppose that $X$ is such a space, and that it can be written as an at most countable union of ringed spaces $(X_p, \widetilde{\mathcal{F}}_p)$. Then, we can form a square/rectangle with the Hochschild cohomologies of each of these spaces and proceed as usual. This cohomology generally depends on the chosen "decomposition" of $X$ as a union. The reason why we consider this second cohomology is mainly because it is related to the previous one. A semi-structural ringed space given by $X$ itself with many different sheaves on it is not the best way to deal with the situation, so we consider the disjoint union of the ringed spaces $(X,\widetilde{\mathcal{F}}_p)$ obtained from the structured Hochschild cohomology for structural ringed spaces:
\begin{equation}\label{Eq:4.1}
\widetilde{X}:= \bigsqcup_p \, (X,\widetilde{\mathcal{F}}_p):= \bigcup_p \, \lbrace (x,p), x \in X \rbrace 
\end{equation}
The spaces $\lbrace (x,p), x \in X \rbrace = X \times \lbrace p \rbrace$ together with $\widetilde{\mathcal{F}}_p$ can still be seen as ringed spaces, for each $p$ (we use the same notation for the sheaf of rings $\widetilde{\mathcal{F}}_p$ on $X$ and for the sheaf of rings on $\lbrace (x,p), x \in X \rbrace$). Thus, the disjoint union above, denoted by $\widetilde{X}$, is a semi-structural ringed space obtained from the first cohomology introduced in this section. If we construct its Hochschild cohomology as semi-structural space, and the order of the $p$'s is the same as the previous cohomology, then we obtain precisely the same cohomology (up to the terms $p$ in the couples).\\
We now turn to schemes; our main references are [24-29]. Recall that a $\mathbb{K}$-scheme $X$ is a scheme together with a morphism from $X$ to $\Spec \mathbb{K}$; a $\mathbb{K}$-scheme is separated if the diagonal is a closed immersion; a $\mathbb{K}$-scheme is of finite type if there exists a finite affine covering $(X_i)_i$ such that the $\mathbb{K}$-algebras $\Gamma(X_i, \mathcal{O}_X)$ are generated by a finite number of elements. Following [25], given a separated $\mathbb{K}$-scheme $X$ of finite type, with $\mathbb{K}$ noetherian ring, the Hochschild cochain complex of $X$ with values in some $\mathcal{O}_X$-module $\mathcal{A}$ is $R \Hom_{\mathcal{O}_{X^2}} (\mathcal{O}_X, \mathcal{A})$ ($R$ indicates the right derived functor). It can be proved that, in some cases, this is equivalent to the definition given by Grothendieck and Loday; see Theorem 2.1 in [27]. All of this can be generalised thanks to structural schemes. Firstly, we roughly think of a structural affine scheme as a structural locally ringed spaces (i.e. a structural ringed space such that all the stalks of its sheaf are structured spaces whose fixed neighborhoods are all local rings) which can be decomposed, via the fixed neighborhoods, into affine schemes. We now make more precise this last part: we have a structural locally ringed space $(X,\mathcal{F})$, which is a structural ringed space (and so we can decompose it as before, obtaining ringed spaces $(X,\mathcal{F}_p)$), and whose stalks $\mathcal{F}_x = \lim_{\longrightarrow_{U \ni x}} \mathcal{F}(U)$ (see [1] for the definition of direct limits for structured spaces) are structured spaces with only local rings as fixed neighborhoods (and hence we can decompose them into local rings $(\mathcal{F}_x)_t$). Now, since direct limits of structured spaces are structured spaces (see Proposition 3.4 in [1]) and the fixed neighborhoods of our direct limit are still local rings (it is known that direct limits of families of local rings are local rings), it is clear that the stalks of $\mathcal{F}$ can be decomposed into stalks $(\mathcal{F}_x)_p$ of $\mathcal{F}_p$, where we can (and will) use \textit{the same} ordering of $p$ as before. Clearly, the spaces $(X,\mathcal{F}_p)$ are locally ringed spaces, because of the previous discussion. If each of these spaces is isomorphic to $\Spec(R_p)$ for some $R_p$, then these spaces are actually affine schemes, and we will call the space considered at the beginning a structural affine scheme. Now that we have made precise the idea above, we can sum up everything in:
\begin{definition}\label{Def:4.4}
A structural affine scheme is a structural locally ringed space such that, after being decomposed via the fixed neighborhoods, each of the locally ringed spaces obtained is isomorphic to $\Spec(R_p)$ for some $R_p$ (equivalently, such that each of the locally ringed spaces obtained is an affine scheme).
\end{definition}
We can now define structural schemes:
\begin{definition}\label{Def:4.5}
A structural scheme is a structural locally ringed space which can be covered by open sets that (when viewed as structural locally ringed spaces) are structural affine schemes.
\end{definition}
It is not difficult to define, via the decomposition into fixed neighborhoods, a notion of separated $\mathbb{K}$-structural scheme $X$ of finite type, with $\mathbb{K}$ a noetherian ring. Then, we can decompose this structural scheme into schemes via the fixed neighborhoods (to do this, consider a covering via structural affine schemes, assuming that all these spaces are such that their sheaves have values in the same category \textbf{Part}$(h, \textbf{Ring})$ (that is, with the same $h$); then, for each $p$, consider the $p$-th affine scheme obtained by decomposing some structural affine scheme in the covering (and decompose this way all the structural affine schemes in the covering). It is clear that the union (for each $p$ fixed) of all the $p$-th affine schemes is a scheme; this is the $p$-th scheme obtained by decomposing the structural scheme), and form the usual square/rectangle with their Hochschild cohomologies. The remaining part of the construction of the structured cohomology is by now well known. Notice that, if we had defined, as for ringed spaces, the notions of 'semi-structural affine scheme' and 'semi-structural scheme', we would have found out that the former is actually equivalent to the concept of scheme and also to the latter. This is why we considered only the above definition. We can summarise everything in:
\begin{theorem}[Structured Hochschild cohomology]\label{Thm:4.1}
Consider a structural ringed space $(X,\mathcal{F})$ (see Definition \ref{Def:4.1}). Decompose as usual the sheaf $\mathcal{F}:X \rightarrow \textbf{Part}(h, \textbf{Ring})$ into $\mathcal{F}_p : X \rightarrow \textbf{Ring}$, which are presheaves of rings that become sheaves after sheafification ($\widetilde{\mathcal{F}}_p$). This way, we obtain ringed spaces $(X,\widetilde{\mathcal{F}}_p)$. We can construct a square/rectangle with their Hochschild cohomologies as usual, and then consider square/rectangular cohomology (or, if the square/rectangle is actually a double complex, we can also consider its cohomology or the cohomology of the total complex derived from it). The cohomology theory obtained is equivalent, up to the terms $p$ in their couples, to the structured Hochschild cohomology for the semi-structural ringed space $\widetilde{X}$ defined by \eqref{Eq:4.1}, where we use the same ordering for the $p$'s of the previous cohomology.\\
Similarly, consider a separated $\mathbb{K}$-structural scheme of finite type (see Definition \ref{Def:4.5}), with $\mathbb{K}$ noetherian ring. This space can be decomposed, via the fixed neighborhoods, into schemes: construct their Hochschild cohomologies and form the square/rectangle as usual. We can then evaluate the square/rectangular (or, if we have a double complex, also the other two available) cohomologies and obtain a structured Hochschild cohomology for the structural scheme.
\end{theorem}
\section{Structured topological K-theory}
We conclude this paper with a generalisation of vector bundles and the corresponding structured cohomology (to be precise, K-theory is a generalised cohomology, so we should say 'structured generalised cohomology'). We mainly refer to [39-46] and [50-55] for K-theory (topological, algebraic, ...), while we refer to [37-38] for spectral sequences and, in particular, for Atiyah-Hirzebruch spectral sequence, which turns out to be a useful tool to evaluate generalised cohomology theories.\\
Here we will use a decomposition method which is a bit different from the previous ones (we do not have sheaves as before), but the idea is always the same. We want to define structural vector bundles in such a way that they can be decomposed into fixed neighborhoods, as in the previous cases, in order to obtain vector bundles which allow us to develop a cohomology theory (this time, however, we will not use square/rectangle cohomology). The key to the "right" definition of structural vector bundles is to consider the homeomorphisms on the fixed neighborhoods (similarly to the isomorphisms in the definition of structural scheme), and not on the entire space. More precisely, begin with two topological spaces $E$ and $X$, and with a continuous surjection $\pi : E \rightarrow X$. Clearly, the fundamental condition here is that $\pi^{-1} ( \lbrace x \rbrace)$ must be a structured space with vector spaces as neighborhoods. To be precise, in this case we can require that $\pi^{-1} ( \lbrace x \rbrace) \in \textbf{Str}(h,\textbf{Vect}_{\mathbb{F}})$ $\forall x \in X$, where $h$ and the field $\mathbb{F}$ are fixed for all the points $x$ in $X$ (\textbf{Vect}$_{\mathbb{F}}$ is the category of vector spaces over $\mathbb{F}$). The fact that $h$ is the same for all $x$ allows us to decompose each $\pi^{-1} ( \lbrace x \rbrace)$ (after reordering) into fixed neighborhoods $(\pi^{-1} ( \lbrace x \rbrace))_p$ (which are nonempty). These are all vector spaces by definition of \textbf{Str}$(h,\textbf{Vect}_{\mathbb{F}})$. We assume, as usual, the the number of $p$'s is at most countably infinite. Suppose that, for each $x \in X$ and $\forall p$, there is some natural number $k$ (which could depend on $x$ and/or $p$), an open neighborhood $V$ of $x$ and a homeomorphism ($(\pi^{-1} ( V ))_p$ is just a notation for $\bigcup_{x \in V} (\pi^{-1} (\lbrace x \rbrace))_p$)
$$ \phi_p : V \times \mathbb{F}^k \rightarrow (\pi^{-1} ( V ))_p$$
such that
$$ (\pi_p \circ \phi_p) (y,v)=y, \, \forall v \in \mathbb{F}^k$$
and the map $v \mapsto \phi_p (y,v)$ is a linear isomorphism of the vector spaces $\mathbb{F}^k$ and $(\pi^{-1} ( \lbrace x \rbrace))_p$. Here, $\pi_p$ is the map defined as follows:
$$\pi_p \equiv \pi|_{E_p} : E_p \rightarrow X$$
where
\begin{equation}\label{Eq:5.1}
E_p:= \bigcup_{x \in X} (\pi^{-1} ( \lbrace x \rbrace))_p
\end{equation}
and as usual we use the definition of composition assuming that the image is \textit{contained}, and not necessarily equal, to the domain. Notice that $(\pi^{-1} (T))_p$ is the same as $\pi^{-1} _p (T)$. Moreover, it is not difficult to see, by definition of $E_p$ and $\pi_p$, that each $\pi_p$ is still a surjection. We need it to be also continuous: we require this in our definition (if it is not, we only need to change the topologies of $E$ and $X$ so that $\pi_p$ (and also $\pi$) are continuous). This way, the spaces $(E_p, X, \pi_p)$ are vector bundles, and thus the "right" definition we were looking for is the following:
\begin{definition}\label{Def:5.1}
A structural vector bundle is a structure $(E,X, \pi)$, with $E$ and $X$ topological spaces and $\pi : E \rightarrow X$ continuous surjection, such that $\pi^{-1} (\lbrace x \rbrace) \in \textbf{Str}(h, \textbf{Vect}_{\mathbb{F}})$ ($h$ and the field $\mathbb{F}$ are fixed for all $x \in X$) and the following conditions are satisfied:\\
(i) After decomposing each $\pi^{-1} (\lbrace x \rbrace)$ into fixed neighborhoods $(\pi^{-1} (\lbrace x \rbrace))_p$, it holds true that: for each $x \in X$ and $\forall p$, there is a natural number $k$ (which may depend on $x$ and/or $p$), an open neighborhood $V$ of $x$ and a homeomorphism:
$$ \phi_p : V \times \mathbb{F}^k \rightarrow (\pi^{-1} ( V ))_p$$
such that 
$$ (\pi_p \circ \phi_p) (y,v)=y, \, \forall v \in \mathbb{F}^k$$
and the mapping $v \mapsto \phi_p (y,v)$ is a linear isomorphism of the vector spaces $\mathbb{F}^k$ and $(\pi^{-1} ( \lbrace x \rbrace))_p$. Here, $\pi_p \equiv \pi|_{E_p} : E_p \rightarrow X$, with $E_p$ as in \eqref{Eq:5.1}, and $(\pi^{-1} ( V ))_p$ is just a notation for $\bigcup_{x \in V} (\pi^{-1} (\lbrace x \rbrace))_p$;\\
(ii) $\forall p$, $\pi_p$ is continuous (changing the topology on $E$ and/or $X$, we can obtain the continuity of $\pi$ and all the $\pi_p$'s).
\end{definition}
\begin{remark}\label{Rm:5.1}
\normalfont Even though it is not required by this general definition, when dealing with structured topological K-theory we will consider some assumptions on the topological space $X$. For instance, as in the usual topological K-theory, $X$ is supposed to be compact and Hausdorff.
\end{remark}
As we have already seen above, given a structural vector bundle, we can decompose it into its fixed neighborhoods, which turn out to give rise to the vector bundles $(E_p, X, \pi_p)$. Now the idea is to first construct $K^0$ in a similar way to what we do in topological K-theory. Consider the space of all the isomorphisms of structural vector bundles over the same $X$ and w.r.t. the same $h$ and $\mathbb{F}$; indicate it with \textbf{StrVect}$(X, h, \mathbb{F})$. The Whitney sum '$\oplus$' of two structural vector bundles (both w.r.t. the same $h$ and the same field) is the structural vector bundle obtained by applying the usual Whitney sum of vector bundles to each of its components (given by the decomposition into fixed neighborhoods). It is then not difficult to see that \textbf{StrVect}$(X, h, \mathbb{F})$, w.r.t. the sum
\begin{equation}\label{Eq:5.2}
[E_1] + [E_2] := [E_1 \oplus E_2]
\end{equation}
is an abelian monoid: this can be checked similarly to the case of vector bundles. We only notice that the class $[0]$ (the identity element of \textbf{StrVect}$(X, h, \mathbb{F})$) is the class of the zero-dimensional trivial structural bundle, i.e. $\bigcup_p (X \times \mathbb{F}^0 \times \lbrace p \rbrace)$, which is clearly a structural vector bundle. Indeed, let $\pi$ be the projection onto the first term, i.e. $\pi(x,0,t):=x$. Then
$$\pi^{-1}(\lbrace x \rbrace)= \bigcup_t \lbrace (x,0,t) \rbrace$$
and hence we have:
$$(\pi^{-1}(\lbrace x \rbrace))_p= (\bigcup_t \lbrace (x,0,t) \rbrace)_p = \lbrace (x,0,p) \rbrace$$
We can define addition and multiplication as follows: $(x,0,t)+(x,0,t):=(x,0+0,t)=(x,0,t)$ and $\alpha (x,0,t):=(x, \alpha 0, t)=(x,0,t)$. Then, this is clearly a vector space, and thus (since the $p$'s are defined using $h$) $\pi^{-1}(\lbrace x \rbrace) \in \textbf{Str}(h, \textbf{Vect}_{\mathbb{F}})$ (\textit{only with structured topological K-theory, we \emph{allow} the possibilty of having as fixed neighborhoods some \emph{singleton} sets (in this case, all the sets have only one point, so they are all singleton); clearly, this adds some objects to the category} $\textbf{Str}(h, \textbf{Vect}_{\mathbb{F}})$ (\textit{and hence also to} \textbf{StrVect}$(X, h, \mathbb{F})$), \textit{which however will still be indicated in the same way (this should not cause confusion)}). It is not difficult to see that the map
$$\phi_p : V \times \mathbb{F}^0 \rightarrow (\pi^{-1}(V))_p = \bigcup_{x \in V} \lbrace (x,0,p) \rbrace$$
defined by:
$$\phi_p(y,0):=(y,0,p)$$
is a homeomorphism: indeed, since $p$ is fixed this is clearly a bijection, and changing (if necessary) the topology on the space we can obtain continuity. Furthermore, we clearly have
$$ (\pi_p \circ \phi_p)(y,0)=\pi_p(y,0,p)=y$$
and the map $0 \mapsto (x,0,p)$ is an isomorphism of the vector spaces $\mathbb{F}^0$ and $\lbrace (x,0,p) \rbrace$ (this follows from a well known result of linear algebra and from the fact that both these vector spaces are zero dimensional). Thus, $\bigcup_p (X \times \mathbb{F}^0 \times \lbrace p \rbrace)$ is a structural vector bundle, and it is also clear that it is the identity element w.r.t. $\oplus$.\\
Now we can apply the usual Grothendieck construction to obtain an abelian group. More precisely, we recall that, given an abelian monoid $M$ (say, under $\cdot$), we can construct an abelian group $M^2/ \sim$, where $\sim$ is defined as follows:
$$ (a_1,a_2) \sim (b_1,b_2) \Leftrightarrow \exists c \in M: a_1 \cdot b_2 \cdot c = a_2 \cdot b_1 \cdot c $$
The operation on $M^2/ \sim$ is defined componentwise:
$$ [(a_1,a_2)] + [(b_1,b_2)] := [(a_1 \cdot a_2, b_1 \cdot b_2)] $$
Returning to our case, we apply Grothendieck completion to \textbf{StrVect}$(X, h, \mathbb{F})$ and we obtain an abelian group, which is denoted by $K^0(X)$ and is the zeroth group of the structured topological K-theory. By using, as in the "usual" topological K-theory, the $n$-th reduced suspensions $\Sigma^n$, we can define the following spaces ($n \geq 0$):
\begin{equation}\label{Eq:5.3}
K^{-n} (X):= \widetilde{K}(\Sigma^n X_{+})
\end{equation}
(for the notation, we refer to [39]). A fundamental result, namely Bott Periodicity Theorem (see [47-49]), allows us (among the other things) to extend to positive integers the definition of the K-theory groups. We conjecture that a similar result still holds in the structured case, and that this indeed allows us to extend the above definition to $n < 0$. We summarise everything in the following:
\begin{theorem}[Structured Topological K-Theory]\label{Thm:5.1}
Let $(X, E, \pi)$ be a structural vector space (see Definition \ref{Def:5.1}). Let \textbf{StrVect}$(X, h, \mathbb{F})$ denote the space of the isomorphisms of structural vector bundles belonging to \textbf{Str}$(h,\textbf{Vect}_{\mathbb{F}})$, where we use the more general definition outlined above. Then, define Whitney sums of structural vector bundles on each component (obtained via the usual decomposition into fixed neighborhoods) and let '$+$' be the operation on \textbf{StrVect}$(X, h, \mathbb{F})$ as given by \eqref{Eq:5.2}. The space is then an abelian monoid, which can be turned into an abelian group using Grothendieck completion. The group obtained this way is the zeroth K-theory group, $K^0(X)$. The groups $K^{-n} (X)$ defined in \eqref{Eq:5.3} are the other ("negative") cohomology groups, defined via the $n$-th reduced suspensions.
\end{theorem}
\begin{conjecture}\label{Conj:5.1}
A similar result to Bott periodicity Theorem also holds in the structured case, and this allows us to extend the definition of K-theory cohomology groups to positive integers.
\end{conjecture}
We suggest a possible attempt to approach Conjecture \ref{Conj:5.1}, which by now is not surprising: decompose into fixed neighborhoods and try to find some connections with the usual Bott Periodicity Theorem. As we have seen in this paper, this could be the "best" way to study this Conjecture.
\section{Conclusion}
In this paper we have continued the study of structured spaces begun in [1] and [2]. In particular, we have constructed some categories for structured spaces, and we have also extended to the 'structured case' various well known concepts. This has led to cohomology theories for the new objects which are somehow related to the corresponding "usual" ones. The main idea behind all these cohomologies is to decompose the new objects into their fixed neighborhoods and then apply the usual theories together with the results in [2]. The unique exception is structured topological K-theory, where we followed a similar construction to the one of topological K-theory, thus not using square/rectangular cohomology as in the previous cases. All these generalisations suggest the possibility to generalise also Bott Periodicity Theorem to the structured case, leading to an extension of the cohomology groups (see Conjecture \ref{Conj:5.1}). Moreover, we also note the possiblity to extend, using similar ideas, other well known notions and their (co)homologies to the structured case.\\
\\
\begin{large}
\textbf{References}
\end{large}
\\
$[1]$ Norman, M. (2020). On structured spaces and their properties. Preprint (arXiv:2003.09240)\\
$[2]$ Norman, M. (2020). Two cohomology theories for structured spaces. Preprint (arXiv:2004.11152)\\
$[3]$ Artin, M.; Grothendieck, A.; Verdier, Jean-Louis. eds. (1972). S\'eminaire de G\'eom\'etrie Alg\'ebrique du Bois Marie - 1963-64 - Th\'eorie des topos et cohomologie \'etale des sch\'emas - (SGA 4) - vol. 2. Lecture Notes in Mathematics (in French). 270. Berlin; New York: Springer-Verlag\\
$[4]$ Gl\"ockner, H. (2007). Direct limits of infinite-dimensional Lie groups compared to
direct limits in related categories. J. Funct. Anal. 245, 19-61\\
$[5]$ Gl\"ockner, H. (2011). Direct limits of infinite-dimensional Lie groups, pp. 243-280 in:
Neeb, K. - H.; Pianzola, A. (Eds.). "Developments and Trends in Infinite Dimensional Lie Theory", Progr. Math. 288, Birkh\"auser, Boston.\\
$[6]$ Gl\"ockner, H. (2019). Direct limits of regular Lie groups. Preprint (arXiv:1902.06329)\\
$[7]$ Gallier, J.; Quaintance, J. (2019). A Gentle Introduction to Homology, Cohomology, and
Sheaf Cohomology, University of Pennsylvania\\
$[8]$ Bosch, S. (2013). Affine Schemes and Basic Constructions. In: Algebraic Geometry and Commutative Algebra. Universitext. Springer, London\\
$[9]$ Nelson, P. (2009). An introduction to schemes. Lecture Notes. University of Chicago, Chicago\\
$[10]$ Ullery, B. (2009). An introduction to affine schemes. Lecture notes. University of Chiacago\\
$[11]$ Mumford, D. (1988). The Red Book of Varieties and Schemes, 2nd edition. Springer-Verlag Berlin Heidelberg\\
$[12]$ An, F. (2011). Affine structures on a ringed space and schemes. Chin. Ann. Math. Ser. B 32, 139-160\\
$[13]$ Mdzinarishvili, L.; Chechelashvili, L. (2015). \v{C}ech Cohomology with Coefficients in a Topological Abelian Group. J Math Sci 211, 40-57\\
$[14]$ Loday-Richaud M. (2016). Sheaves and \v{C}ech Cohomology with an Insight into Asymptotics. In: Divergent Series, Summability and Resurgence II. Lecture Notes in Mathematics, vol 2154. Springer, Cham\\
$[15]$ Godement, R. (1958). Topologie Alg\'ebrique et Th\'eorie des Faisceaux. Hermann, Paris\\
$[16]$ Lowen, W. T.; Van den Bergh, M. (2005). Hochschild cohomology of abelian categories and ringed
spaces, Adv. in Math. 198, no. 1, 172-222\\
$[17]$ Ayala, D.; Francis, J. (2015). Factorization homology of topological manifolds, Journal of Topology, Volume 8, Issue 4, 1045-1084\\
$[18]$  Hochschild, G. (1945). On the cohomology groups of an associative algebra. Ann. Math., 46(2): 58-67\\
$[19]$ Kr\"ahmer, U. (2007). Poincar\'e duality in Hochschild (co)homology. New Techniques in Hopf Algebras and Graded Ring Theory, K. Vlaam. Acad. Belgie Wet Kunsten, Brussels, pages 117-125\\
$[20]$ Wahl, N.; Westerland, C. (2015). Hochschild homology of structured algebras. Preprint (arXiv:1110.0651)\\
$[21]$ Keller, B. (1998). On the cyclic homology of ringed spaces and schemes, Doc. Math. 3, 231-259
(electronic)\\
$[22]$ Loday, J.-L. (1986). Cyclic homology, a survey, Geometric and algebraic topology, Banach Center
Publ., vol. 18, PWN, Warsaw, 281-303\\
$[23]$ Mitchell, B. (1972). Rings with several objects, Advances in Math. 8, 1-161\\
$[24]$ Weibel, C. (1996). Cyclic homology for schemes, Proc. Amer. Math. Soc. 124, no. 6, 1655-
1662\\
$[25]$ Yekutieli, A. (2002). The continuous Hochschild cochain complex of a scheme, Canad. J. Math. 54, no. 6, 1319-1337\\
$[26]$ Artin, M.; Grothendieck, A.; Verdier, J. L. (1963). Theorie des topos et cohomologie \'etale des
sch\'emas, SGA4, Tome 1, Lecture Notes in Mathematics, vol. 269, Springer Verlag\\
$[27]$ Swan, R. G. (1996). Hochschild cohomology of quasiprojective schemes. Journal of Pure and Applied Algebra, Volume 110, Issue 1, 57-80\\
$[28]$ Banerjee, A. (2011). A note on product structures on Hochschild homology of schemes. Colloquium Mathematicae 123.2: 233-238\\
$[29]$ Gerstenhaber, M.; Schack, S. (1988). Algebraic cohomology and deformation theory. In:
Deformation Theory of Algebras and Structures and Applications (Il Ciocco, 1986), NATO Adv. Sci. Inst. Ser. C Math. Phys. Sci. 247, Kluwer, Dordrecht, 11-264\\
$[30]$ Roitzheim, C.; Whitehouse, S. (2011). Uniqueness of $A_{\infty}$-structures and Hochschild cohomology. Algebr. Geom. Topol. 11, no. 1, 107-143\\
$[31]$ Ginot, G.; Tradler, T.; Zeinalian, M. (2014). Higher Hochschild homology, topological chiral homology and factorization algebras, Comm. Math. Phys. 326, no. 3, 635-686\\
$[32]$  Ginot, G.; Tradler, T.; Zeinalian, M. (2014). Higher Hochschild cohomology, Brane topology and
centralizers of $E_n$-algebra maps. Preprint (arXiv:1205.7056)\\
$[33]$ Tamarkin, D.; Tsygan, B. (2001). Cyclic Formality and Index Theorems. Letters in Mathematical Physics 56, 85-97\\
$[34]$ Wiegand, R.; Wiegand, S. (2010). Prime Ideals in Noetherian Rings: A Survey. In: Albu T., Birkenmeier G.F., Erdo\u{g}gan A., Tercan A. (eds) Ring and Module Theory. Trends in Mathematics. Springer, Basel\\
$[35]$ Wiegand, R.; Wiegand, S. (1976). The maximal ideal space of a Noetherian ring. J. Pure Appl. Algebra 8, 129-141\\
$[36]$ Heitmann, R. (1977). Prime ideal posets in Noetherian rings. Rocky Mountain J. Math. 7, 667-673\\
$[37]$ Husem\"oller, D.; Joachim, M.; Jur\v{c}o, B.; Schottenloher M. (2008). The Atiyah-Hirzebruch Spectral Sequence in K-Theory. In: Basic Bundle Theory and K-Cohomology Invariants. Lecture Notes in Physics, vol 726. Springer, Berlin, Heidelberg\\
$[38]$ McCleary, J. (2000). A User's Guide to Spectral Sequences (Cambridge Studies in Advanced Mathematics). Cambridge: Cambridge University Press\\
$[39]$ Friedlander, E. M. (2008). An introduction to K-theory. In: Some recent developments in
algebraic K-theory, volume 23 of ICTP Lect. Notes, pages 1-77. Abdus Salam Int. Cent. Theoret. Phys., Trieste\\
$[40]$ Rosenberg, J. (1994). Algebraic K-Theory and Its Applications. Graduate texts in Mathematics. Springer-Verlag, New York\\
$[41]$ Weibel, C. (2013). The K-book: an introduction to algebraic K-theory. Grad. Studies in Math. 145. American Math Society\\
$[42]$ Atiyah, M. F. (1989). K-theory. Advanced Book Classics (2nd ed.). Addison-Wesley\\
$[43]$ Karoubi, M. (1978). K-theory: an introduction. Classics in Mathematics. Springer-Verlag\\
$[44]$ Atiyah, M. F.; Hirzebruch, F. (1961). Vector bundles and homogeneous spaces, Proc. Sympos. Pure Math., 3, American Mathematical Society, 7-38\\
$[45]$ Quillen, D. (1975). Higher algebraic K-theory, Proc. Intern. Congress Math., Vancouver, 1974, vol. I, Canad. Math. Soc., 171-176\\
$[46]$ Swan, R. (1963). The Grothendieck ring of a finite group. Topology, 2, 85-110\\
$[47]$ Giffen, C. H. (1996). Bott periodicity and the Q-construction, Contemp. Math. 199, 107-124\\
$[48]$ Bott, Ra. (1956). An application of the Morse theory to the topology of Lie-groups. Bulletin de la Soci\'et\'e Math\'ematique de France, 84: 251-281\\
$[49]$ Bott, R. (1957). The stable homotopy of the classical groups. Proceedings of the National Academy of Sciences of the United States of America, 43 (10): 933-935\\
$[50]$ Gillet, H. (1992) Comparing algebraic and topological K-theory. In: Higher Algebraic K-Theory: an overview. Lecture Notes in Mathematics, vol 1491. Springer, Berlin, Heidelberg\\
$[51]$ Ausoni, C. (2010). On the algebraic K-theory of the complex K-theory spectrum. Invent. Math., 180(3): 611-668\\
$[52]$ Ausoni, C.; Rognes, J. (2002). Algebraic K-theory of topological K-theory. Acta Math., 188(1): 1-39\\
$[53]$ Ausoni, C. (2005). Topological Hochschild homology of connective complex K-theory. Am. J. Math. 127(6), 1261-1313\\
$[54]$ Blumberg, A. J.; Mandell, M. A. (2008). The localization sequence for the algebraic K-theory of topological K-theory. Acta Math., 200(2): 155-179\\
$[55]$ Krishna, A.; Ravi, C. (2017). Equivariant vector bundles, their derived category and K-theory on affine schemes. Annals of K-theory. Vol. 2, No. 2, 235-275\\
$[56]$ Bredon, G. E. (1997). Sheaf theory. Graduate Texts in Mathematics, 170 (2nd ed.), Berlin, New York: Springer-Verlag

\end{document}